\documentclass[10 pt]{article}
\usepackage{amssymb,amsthm}
\font\bigmath=cmsy10 scaled \magstep 2
\newcommand{\bigtimes}{\hbox{\bigmath \char'2}}
\newcommand{\cchi}{\raise 2 pt \hbox{$\chi$}}
\newcommand{\ggamma}{\raise 2 pt \hbox{$\gamma$}}
\newcommand{\vvarphi}{\raise 2 pt \hbox{$\varphi$}}
\newcommand{\cl}{c\ell}
\newcommand{\emp}{\emptyset}
\newcommand{\ber}{\mathbb R}
\newcommand{\ben}{\mathbb N}
\newcommand{\bez}{\mathbb Z}
 
\newcommand{\nhat}[1]{\{1,2,\ldots,#1\}}
\newcommand{\pf}{{\mathcal P}_f}

\newtheorem{theorem}{Theorem}[section]
\newtheorem{corollary}[theorem]{Corollary}
\newtheorem{lemma}[theorem]{Lemma}
\newtheorem{question}[theorem]{Question}

\theoremstyle{definition}
\newtheorem{definition}[theorem]{Definition}

\title{F\o lner, Banach, and translation density are equal and other new results about density in left
amenable semigroups}

\date{}

\author{Daniel Glasscock
\footnote{Department of Mathematics and Statistics, 
University of Massachusetts Lowell, Southwick Hall 301D,
1 University Ave, Lowell, MA 01854, USA. This author gratefully
acknowledges support from the National Science Foundation under
Grant No.\ DMS-2418589. {\tt daniel\char'137 glasscock@uml.edu}}
\and
Neil Hindman
        \footnote{Department of Mathematics,
                 Howard University,
                  Washington, DC 20059, USA.\hfill\break
                  Corresponding Author. {\tt nhindman@aol.com}}
        \and
Dona Strauss
        \footnote{95 Lowther Rd, Brighton BN16LH, England.\hfill\break
{\tt donastrauss@gmail.com}\hfill\break
{Keywords: Strong F\o lner Condition, F\o lner density, Left invariant means, Banach density\hfill\break
MSC[2020] 20M10, 05C42}}
}

\begin{document}
\maketitle

\begin{abstract}In any semigroup $S$ satisfying the {\it Strong F\o lner Condition\/},
there are three natural notions of density for a subset $A$ of $S$: F\o lner density $d(A)$,
Banach density $d^*(A)$, and translation density $d_t(A)$.
If $S$ is commutative or left cancellative, it is known that these three
notions coincide.  We shall show that these notions coincide for every semigroup $S$ which 
satisfies the Strong F\o lner Condition.  We solve a problem that has
been open for decades, showing that if $S$ is left amenable, the set of ultrafilters every member of which has
positive Banach density is a two sided ideal of $\beta S$. 
We  investigate the density properties of subsets of $S$ in the case in which the minimal left ideals
of $\beta S$ are singletons. This occurs in 
 all semilattices and all semigroups which have a right zero. We show that this is equivalent to the statement that $S$ satisfies 
$SFC$ and that, for every subset $A$ of $S$, $d(A)\in \{0,1\}$.  We also examine the relation between the density properties 
of two semigroups when one is a quotient of the other.  If $S$ satisfies $SFC$,
we show that an arbitrary F\o lner  net in $S$
determines the density of all of the subsets of $S$.  And we prove that, if $S$ and $T$ are left amenable semigroups, then
 $d^*(A\times B)=d^*(A)d^*(B)$ for every subset $A$ of $S$ and every subset $B$ of $T$. \end{abstract}

\section{Three notions of density in semigroups}

We begin by introducing some of the notions that we are concerned with here.
Given a set $X$, we let $\pf(X)$ be the set of finite nonempty subsets
of $X$. If $(S,\cdot)$ is a semigroup and $x \in S$, we denote left and right multiplication by $x$ by $\lambda_x: S \to S$
and $\rho_x: S \to S$, respectively.  For $A\subseteq S$, we define
$x^{-1}A = \lambda_x^{-1}[A]=\{y\in S:xy\in A\}$ and $Ax^{-1} = \rho_x^{-1}[A]=\{y\in S: yx \in A\}$.

In the following subsections, we treat three notions of density in semigroups: Banach density $d^*$, F\o{}lner
density $d$, and translation density $d_t$.

\subsection{Banach density}

Let $(S,\cdot)$ be a semigroup. Let  $l_{\infty}(S)$ be the set of bounded real valued
functions on $S$ with the supremum norm, denoted by $\Vert\hskip 5 pt\Vert_{\infty}$.
Let $l_{\infty}(S)^*$ be the set of continuous real valued linear functionals on $l_\infty(S)$
with the dual norm $|| \mu || = \sup \{ \mu(f) : || f ||_{\infty} \leq 1 \}$.
 A {\it mean\/} on $S$ is an element of $l_{\infty}(S)^*$ such
that $||\mu||=1$ and $\mu \geq 0$, that is, whenever $g\in l_{\infty}(S)$ and for all $s\in S$,
$g(s)\geq 0$, one has that $\mu(g)\geq 0$. A {\it left invariant mean\/} on $S$ is a mean $\mu$ such
that for all $s\in S$ and all $g\in l_{\infty}(S)$, $\mu(g\circ \lambda_s)=\mu(g)$. 
The semigroup $S$ is 
defined to be {\it left amenable\/} if and only if there exists a 
left invariant mean on $S$. 

We denote by $MN(S)$ and $LIM(S)$ the set of means and left invariant means on $S$,
respectively. If $\mu$ is a mean on $S$ and $A \subseteq S$,
it is useful in discussing density to use
$\mu(A)$ to denote $\mu(\cchi_A)$,
where $\cchi_A$ is the characteristic function of $A$, 
because density is a property of sets.

The weak* topology of $l_{\infty}(S)^*$ is defined by stating that a net $\langle\mu_{\alpha}\rangle_{\alpha\in D}$ 
converges to a limit $\mu$ in this space if and only if the net $\langle\mu_{\alpha}(f)\rangle_{\alpha\in D}$ converges to
$\mu(f)$ in $\ber$ for every $f\in l_{\infty}(S)$. The weak* topology is the restriction to $l_{\infty}(S)^*$ of the product topology on $\bigtimes_{g\in l_\infty(S)}\,\ber$.
By the Alaoglu Theorem \cite[Theorem B25]{HR}, the closed unit ball of $l_{\infty}(S)^*$ is compact in the weak* topology.

The following notion of density is defined in any left 
amenable semigroup. Following \cite{BG}, we call it {\it Banach density\/}.

\begin{definition} \label{defBdensity} Let $(S,\cdot)$ be a left amenable semigroup,
and let $A\subseteq S$. The {\it Banach density of $A$\/}
is defined by $d^*(A)=\sup\{\lambda(\cchi_A):\lambda$ is a left invariant
mean on $S\}$.
\end{definition}

In the remainder of this subsection, we collect some preliminary results that will be
key to relating Banach density to other notions of density later on.

\begin{lemma}\label{lemlambdaright} Let $(S,\cdot)$ be a left amenable semigroup,
let $\lambda$ be a left invariant mean on $S$, and let $R$ be a right 
ideal of $S$. Then $\lambda(\cchi_R)=1$.\end{lemma}

\begin{proof} Pick $a\in R$. Then $\cchi_R\circ\lambda_a=\cchi_S$
so $\lambda(\cchi_R)=\lambda(\cchi_R\circ\lambda_a)=\lambda(\cchi_S)=1$.
\end{proof} 

\begin{lemma}\label{inMNS} Let $(S,\cdot)$ be a semigroup, and let $\mu\in l_{\infty}(S)^*$. 
If $\mu\geq 0$, then $\mu\in MN(S)$ if and only if $\mu(S)=1$. \end{lemma}

\begin{proof} Since $\cchi_S$ is the maximum element of the unit ball of $l_{\infty}(S)$, 
for any  $\mu\geq 0$ in $l_{\infty}(S)^*$,  $||\mu||=\mu(S)$. Thus $||\mu||=1$ if and only if $\mu(S)=1$.   \end{proof}

In the following lemmas, we make use of the linear subspace $E$ of $l_{\infty}(S)$ generated
by the functions of the form $\cchi_A$, where $A$ denotes a subset of $S$.  Note that this is
precisely the subset of $l_{\infty}(S)$ consisting of those functions with finite range.

\begin{lemma}\label{Edense} Let $(S,\cdot)$ be a semigroup, and let $E$ denote the linear subspace
of $l_{\infty}(S)$ generated by the functions of the form $\cchi_A$, where $A$ denotes a subset of $S$.
Then $E$ is uniformly dense in $l_{\infty}(S)$. \end{lemma}

\begin{proof} This is easy to prove directly
by an elementary argument.
It also follows from the Stone-Weierstrass Theorem, because 
$\{\widetilde f:f\in E\}$ is a subalgebra of $C(\beta S)$ which 
separates points and contains the constant function $\cchi_{\beta S}$. 
(Here $\widetilde f:\beta S\to\ber$ is the continuous extension of $f$.)\end{proof}

\begin{lemma}\label{content} Let $(S,\cdot)$ be a semigroup, and 
let $\mu: \mathcal{P}(S) \to \ber$ be a non-negative, finitely additive set function such that $\mu(S)=1$.
Then $\mu$ extends uniquely to a mean on $S$. \end{lemma}

\begin{proof}
Denote by $E$ the subset of $l_\infty(S)$ consisting of those functions with finite range.  By Lemma
\ref{Edense}, the set $E$ is uniformly dense in $l_\infty(S)$.
    
We define a function $\nu: E \to \ber$ as follows: for $f \in E$,
\[\nu(f) = \sum_{x \in \textrm{\footnotesize Range}(f)} x \mu \big(f^{-1}[\{x\}] \big).\]
Note that $\nu$ extends $\mu$ in the sense that for all $A \subseteq S$, 
$\nu(\cchi_A) = \mu(A)$. By the finite additivity of $\mu$, it is easy to 
see that for all $n \in \ben$, all $c_1, \ldots, c_n \in \ber$, and all 
pairwise disjoint $A_1, \ldots, A_n \subseteq S$, $\nu \left( \sum_{i=1}^n c_i \cchi_{A_i}\right) 
= \sum_{i=1}^n c_i \mu(A_i)$.  Now we note that

\hbox to \hsize{(*)\hfill if $f\in E$, $a>0$, and Range$(f)\subseteq[-a,a]$, then $|\nu(f)|\leq a$.\hfill}

\noindent To see this, let such $f$ and $a$ be given. Note that since 
$\{f^{-1}[\{x\}]:x\in \hbox{Range}(f)\}$ is a set of pairwise disjoint sets and
$\mu(S)=1$, $\sum_{x \in \textrm{\footnotesize Range}(f)} \mu\big(f^{-1}[\{x\}] \big)\leq 1$.
Then $|\nu(f)|\leq \sum_{x \in \textrm{\footnotesize Range}(f)} |x| \mu \big(f^{-1}[\{x\}] \big)
\leq a \sum_{x \in \textrm{\footnotesize Range}(f)} \mu \big(f^{-1}[\{x\}] \big)\leq a$, as required.

We claim that for all $f, g \in E$ and $c \in \ber$, $\nu(cf) = c\nu(f)$ and $\nu (f + g) = \nu (f) + \nu (g)$.  Indeed, the first is immediate. To see the second, suppose $\textrm{Range}(f) = \{c_1, \ldots, c_n\}$ and $\textrm{Range}(g) = \{d_1, \ldots, d_m\}$.  Let $A_{i,j} = \{s \in S \ | \ f(s) = c_i \textrm{ and } g(s) = d_j\}$.  Note that $f = \sum_{i=1}^n \sum_{j=1}^m c_i \cchi_{A_{i,j}}$, $g = \sum_{i=1}^n \sum_{j=1}^m d_j \cchi_{A_{i,j}}$, and that the $A_{i,j}$'s are pairwise disjoint.  It follows that
$$ \sum_{i=1}^n \sum_{j=1}^m (c_i + d_j) \mu(A_{i,j}) = \sum_{i=1}^n \sum_{j=1}^m c_i \mu(A_{i,j}) + \sum_{i=1}^n \sum_{j=1}^m d_j \mu(A_{i,j}).$$
The left hand side is $\nu(f+g)$ and the right hand side is $\nu(f) + \nu(g)$, whence $\nu(f+g) = \nu(f) + \nu(g)$.

We claim that $\nu: E \to \ber$ is uniformly continuous. 
Indeed, suppose $f, g \in E$ are such that $\|f - g\|_\infty < \epsilon$.  
By the definition of $|| \cdot ||_{\infty}$, the range of $f-g$ is contained 
in the interval $[-\epsilon,\epsilon]$. Since $\mu(S) = 1$, we see from
(*) that $|\nu(f) - \nu(g)| = |\nu(f-g)| < \epsilon$.

Because $\nu$ is uniformly continuous and $E$ is dense, the function $\nu$ extends uniquely to a uniformly continuous function $\nu: l_\infty(S) \to \ber$. Clearly $\nu(\cchi_S) = 1$ and $\nu \geq 0$.  For $c \in \ber$, the functions $f \mapsto \nu(cf)$ and $f \mapsto c \nu(f)$ are continuous and agree on $E$, hence are equal.  Similarly, the functions $(f,g) \mapsto \nu(f+g)$ and $(f,g) \mapsto \nu(f) + \nu(g)$ are continuous and agree on $E$, hence are equal.  Therefore, $\nu$ is linear, and hence is a mean on $S$.

It is easy to see that any mean extending $\mu$ must be equal to $\nu$ when restricted to the set $E$.  It follows by the work above, then, that $\nu$ is the unique mean on $S$ extending $\mu$.
\end{proof}

We remind the reader that, if $s\in S$ and $A\subseteq S$, $s^{-1}A$ denotes $\{t\in S:st\in A\}$.

\begin{lemma}\label{s^{-1}A} Let $(S,\cdot)$ be a semigroup, and let $\mu\in MN(S)$. Then $\mu\in LIM(S)$ if and only if
$\mu(s^{-1}A)=\mu(A)$ for every $s\in S$ and every $A\subseteq S$.
\end{lemma}

\begin{proof}
Observe that $\cchi_{s^{-1}A}=\cchi_A\circ\lambda_s$ for every $s\in S$ and every $A\subseteq S$. So, if $\mu\in LIM(S)$,
$\mu(s^{-1}A)=\mu(A)$ for every $s\in S$ and every $A\subseteq S$.

To prove the converse, assume that $\mu$ is  a mean on $S$ with the property that $\mu(s^{-1}A)=\mu(A)$ for every $s\in S$
and every $A\subseteq S$. 

For $s\in S$, define $\tau_s:l_{\infty}(S)^*\to l_{\infty}(S)^*$ by, for $f\in l_{\infty}(S)$,
$\tau_s(\eta)(f)=\eta(f\circ\lambda_s)$. Note that $\tau_s$ is a linear map. We
claim that $\tau_s$ is continuous for the norm topology on $l_{\infty}(S)^*$. 
By \cite[Theorem B.10]{HR}, it suffices to show that $\tau_s$ is bounded. So let
$\eta\in l_{\infty}(S)^*$. We claim that $||\tau_s(\eta)|| \leq ||\eta||$, for which it suffices 
to let $f\in l_{\infty}(S)$ and note that $||f\circ\lambda_s||_{\infty} \leq ||f||_{\infty}$, whereby
$|\tau_s(\eta)(f)|=|\eta(f\circ\lambda_s)| \leq |\eta(f)|$.

Let $E$ denote the linear subspace of $l_{\infty}(S)$ generated by the functions  
$\cchi_A$, where $A$ denotes a subset of $S$, which is uniformly 
dense in $l_{\infty}(S)$ by Lemma \ref{Edense}.
Given $s\in S$ and $A\subseteq S$, we have noted that $\cchi_{s^{-1}A}=\cchi_A\circ\lambda_s$,
so that $\tau_s(\mu)(\cchi_A)=\mu(\cchi_A\circ\lambda_s)=\mu(\cchi_{s^{-1}A})=\mu(\cchi_A)$.
This implies that $\tau_s(\mu)(f)=\mu(f)$ for every $f\in E$. Since $E$ is dense in 
$l_\infty(S)$ and $\tau_s$ is continuous, this implies that $\mu(f\circ\lambda_s)=\mu(f)$ for every 
$f\in l_\infty(S)$, so that $\mu\in LIM(S)$.\end{proof}

\subsection{F\o{}lner density}

The following conditions provide a way to understand left amenability in terms of sets and their images under left multiplication.

\begin{definition}\label{defSFC} Let $(S,\cdot)$ be a semigroup. 
\begin{itemize}
\item[(a)] The semigroup
$S$ satisfies the {\it F\o lner Condition\/} (FC) if and only 
if\hfill\break 
$\big(\forall H\in\pf(S)\big)(\forall \epsilon>0)\big(\exists K\in\pf(S)\big)
(\forall s\in H)(|sK\setminus K|<\epsilon\cdot |K|)$.
\item[(b)] The semigroup $S$ satisfies the {\it Strong F\o lner Condition\/} (SFC) if and only 
if $\big(\forall H\in\pf(S)\big)(\forall \epsilon>0)\big(\exists K\in\pf(S)\big)
(\forall s\in H)(|K\setminus sK|<\epsilon\cdot |K|)$.
\end{itemize}
\end{definition}

The Strong F\o{}lner Condition leads to a natural notion of density.

\begin{definition}\label{defdensity} Let $(S,\cdot)$ be a semigroup
which satisfies SFC, and let $A\subseteq S$. The {\it F\o lner density of $A$\/}
is defined by\hfill\break 
$d(A)=\sup\{\alpha\in[0,1]:\big(\forall H\in\pf(S)\big)
(\forall\epsilon>0)\big(\exists K\in\pf(S)\big)\hfill\break
\big((\forall s\in H)(|K\setminus sK|<\epsilon\cdot|K|)
\hbox{ and }|A\cap K|\geq \alpha\cdot|K|\big)\}$.
\end{definition}

The reader is referred to \cite[Section 4.22]{P} for a readable discussion of
the relationship among FC, SFC, and left amenability and relevant references. 
In particular, any left amenable semigroup satisfies FC and any semigroup 
satisfying SFC is left amenable. By \cite[Theorem 4]{AW}, any commutative semigroup
satisfies SFC and by (the left-right switch of) \cite[Corollary 3.6]{N}, any
left amenble and left cancellative semigroup satisfies SFC.

\begin{definition}\label{defFolnernet} Let $(S,\cdot)$ be a semigroup.
A {\it F\o lner net\/} in $\pf(S)$ is a net $\langle F_\alpha\rangle_{\alpha\in D}$ such that
for each $s\in S$, $\displaystyle\lim_{\alpha\in D}\frac{|F_\alpha\setminus sF_\alpha|}{|F_\alpha|} =0$.
\end{definition}

Of course, a {\it F\o lner sequence\/} is a F\o lner net in which the relevant directed
set is the set $\ben$ of positive integers, with its usual order.  It is immediate that if there exists
a F\o lner net in $\pf(S)$, then $S$ satisfies SFC. 
It is a consequence of Theorem \ref{thmneteq} below that the converse holds.

\begin{definition}\label{defmuF} Let $(S,\cdot)$ be a semigroup.
\begin{itemize}
\item[(1)] For $F\in\pf(S)$, $\mu_F\in l_\infty(S)^*$ is defined by, 
for $g\in l_\infty(S)$, $\mu_F(g)=\frac{1}{|F|}\sum_{t\in F}g(t)$.
\item[(2)] $LIM_0(S)=\{\nu:$ there exists a F\o lner net $\langle F_\alpha\rangle_{\alpha\in D}$ in
$\pf(S)$ such that $\nu$ is a cluster point of the net 
$\langle \mu_{F_\alpha}\rangle_{\alpha\in D}$ in the weak* topology of $l_{\infty}(S)^*$\}.
\end{itemize}
\end{definition}

As the notation suggests, elements of $LIM_0(S)$ are, in fact, left invariant means.  This is recorded
in the following lemma.

\begin{lemma}\label{lemLIM0LIM} Let $(S,\cdot)$ satisfy SFC. Then $LIM_0(S)\subseteq LIM(S)$. 
\end{lemma} 

\begin{proof} This is shown in \cite[Lemma 2.2]{HSc}.\end{proof}

\begin{theorem}\label {weakstarcompact} Let $(S,\cdot)$ be a semigroup satisfying SFC.
Then $LIM(S)$ is convex, and $LIM(S)$ and $LIM_0(S)$ are both compact  in the weak* topology of
$l_{\infty}(S)^*$. \end{theorem}

\begin{proof} It is easy to see that $LIM(S)$ is convex and weak* compact. 
We shall show that $LIM_0(S)$ is weak* compact.

For every $H\in {\mathcal P}_f(S)$ and every $\epsilon>0$, let 
$${\mathcal F}_{H,\epsilon}=\{F\in {\mathcal P}_f(S):(\forall s\in H)(|F\setminus sF|<\epsilon\cdot|F|)\}\,,$$
and let $\mu_{H,\epsilon}=\{\mu_F:F\in {\mathcal F}_{H,\epsilon}\}$.
We shall show that $$\textstyle LIM_0(S)=\bigcap\{\cl \mu_{H,\epsilon}:H\in\pf(S)\hbox{ and }\epsilon>0\}\,.$$ 
This will suffice since each $\mu_F$ is a mean on $S$.
First, let $\nu\in LIM_0(S)$ and pick a F\o lner net $\langle F_\alpha\rangle_{\alpha\in D}$
in $\pf(S)$ such that $\nu$ is a cluster point of $\langle \mu_{F_\alpha}\rangle_{\alpha\in D}$ in
$l_\infty(S)^*$. Let $H\in\pf(S)$ and $\epsilon>0$. To see that $\nu\in\cl \mu_{H,\epsilon}$,
let $U$ be a neighborhood of $\nu$ in $l_\infty(S)^*$.  Pick $\ggamma\in D$ such that for 
all $s\in H$ and all $\alpha\geq \ggamma$ in $D$, $|F_\alpha\setminus sF_\alpha|<\epsilon\cdot|F_\alpha|$
and pick $\alpha\geq \ggamma$ such that $\mu_{F_\alpha}\in U$. Then $\mu_{F_\alpha}\in U\cap\mu_{H,\epsilon}$.

Now let $\nu\in \bigcap\{\cl \mu_{H,\epsilon}:H\in\pf(S)$ and $\epsilon>0\}$.
Let ${\mathcal U}$ be the set of open neighborhoods of $\nu$ in $l_\infty(S)^*$.
Let $D={\mathcal U}\times \pf(S)\times (0,1)$ and direct $D$ by
$(U,H,\epsilon)\leq (V,K,\delta)$ provided $V\subseteq U$, $H\subseteq K$, and $\delta\leq \epsilon$.
For $\alpha=(U,H,\epsilon)\in D,$ pick $F_\alpha\in {\mathcal F}_{H,\epsilon}$ such that 
$\mu_{F_\alpha}\in U\cap \mu_{H,\epsilon}$. Then
$\langle F_\alpha\rangle_{\alpha\in D}$ is a F\o lner net in $\pf(S)$ and $\langle \mu_{F_\alpha}\rangle_{\alpha\in D}$ 
converges to $\nu$.
\end{proof}

\begin{theorem}\label{thmneteq} Let $(S,\cdot)$ be a semigroup satisfying SFC, let
$A\subseteq S$, and let $\delta=d(A)$. There is a F\o lner net
$\langle F_\alpha\rangle_{\alpha\in D}$ such that the net
$\displaystyle \left\langle \frac{|F_\alpha\cap A|}{|F_\alpha|}\right\rangle_{\alpha\in D}$
converges to $\delta$. 

If $\nu$ is any cluster point of the net 
$\langle \mu_{F_\alpha}\rangle_{\alpha\in D}$ in $\bigtimes_{f\in l_\infty(S)}[-||f||_\infty,||f||_\infty]$,
then $\nu\in LIM_0(S)$ and $\nu(\cchi_A)=\delta$. In particular $d(A)\leq d^*(A)$.\end{theorem}

\begin{proof}  Since $d(A)=\delta$, it is a routine exercise to show that\hfill\break
$\big(\forall H\in\pf(S)\big)
(\forall\epsilon>0)\big(\exists K\in\pf(S)\big)
\big((\forall s\in H)(|K\setminus sK|<\epsilon\cdot|K|)
\hbox{ and }\hfill\break(\delta-\epsilon)\cdot |K|<|A\cap K|<(\delta+\epsilon)\cdot|K|\big)$.

Let $D=\pf(S)\times\ben$, and direct $D$ by $(H,n)\leq (K,m)$ if and only if $H\subseteq K$ and $n\leq m$.
For $\alpha=(H,n)\in D$, pick $F_\alpha\in\pf(S)$ such that\\
$(\forall s\in H)(|F_\alpha\setminus sF_\alpha|<{1\over n}\cdot |F_\alpha|$ and
 $(\delta-\frac{1}{n})\cdot|F_\alpha|<|F_\alpha\cap A|<(\delta+\frac{1}{n})\cdot|F_\alpha|$.

Let $\nu$ be a cluster point of the net 
$\langle \mu_{F_\alpha}\rangle_{\alpha\in D}$.
Since $\langle \mu_{F_\alpha}(\cchi_A)\rangle_{\alpha\in D}$ converges
to $\delta$, we have that $\nu(\cchi_A)=\delta$.
\end{proof}

\begin{theorem}\label{thmmax} Let $(S,\cdot)$ be a semigroup satisfying SFC, and let $A\subseteq S$.
Then $d(A)=\max\{\nu(\cchi_A):\nu\in LIM_0(S)\}$.\end{theorem}

\begin{proof} By Theorem \ref{thmneteq}, it suffices to show that if $\nu\in LIM_0(S)$, then
$\nu(\cchi_A)\leq d(A)$. So let $\nu\in LIM_0(S)$ and suppose that 
$\nu(\cchi_A)=\delta>\gamma=d(A)$.  Pick a F\o lner net $\langle F_a\rangle_{a\in D}$ in 
$\pf(S)$ such that $\nu$ is a cluster point of the net
$\langle \mu_{F_a}\rangle_{a\in D}$ in $\bigtimes_{f\in l_\infty(S)}[-||f||_\infty,||f||_\infty]$.

Let $\beta=\delta-\gamma$. We shall show that $d(A)\geq \delta-\frac{\beta}{2}$, a contradiction.
So let $H\in\pf(S)$ and $\epsilon>0$ be given. Since $\langle F_a\rangle_{a\in D}$ is 
a F\o lner net, pick $a\in D$ such that for every $b\in D$ with $b\geq a$ and every $s\in H$,
$\displaystyle \frac{|F_b\setminus sF_b|}{|F_b|}<\epsilon$.  Let $g=\cchi_A$ and let
$U=\pi_{g}^{-1}[(\delta-\frac{\beta}{2},\delta+\frac{\beta}{2})]$. Then
$U$ is a neighborhood of $\nu$, so pick $b\in D$ such that $b\geq a$ and $\mu_{F_b}\in U$.
Then $\displaystyle \frac{|A\cap F_b|}{|F_b|}=\mu_{F_b}(g)>\delta-\frac{\beta}{2}$.  \end{proof}

\begin{question}
Let $(S,\cdot)$ be a semigroup satisfying SFC.  Is $LIM(S)$ the weak* closed convex hull of $LIM_0(S)$?
\end{question} 

We know that the answer is affirmative if $S$ is left cancellative. This follows from the Krein-Milman Theorem and \cite[Corollary 2.13]{HSa}.

\subsection{Translation density}

We will consider a third notion of density, denoted by $d_t$, which is defined in 
any semigroup. (The $t$ stands for {\it translate\/} which seems appropriate 
if the operation is written additively (so $As^{-1}=\{x\in S:xs\in A\}$
in the definition becomes $A-s=\{x\in S:x+s\in A\}$).) This notion of density was
defined first in \cite[Theorem 3.2]{BG} but appeared and was considered in relation
to the upper Banach density in various other places: in \cite[Lemma 9.6]{FK} for $(\ber^d,+)$,
in \cite[Corollary 9.2]{Gri} for $(\bez,+)$, and in a pre-publication version of \cite{JR} (as Theorem G)
for cancellative semigroups satisfying SFC.

\begin{definition}\label{defdt} Let $(S,\cdot)$ be a semigroup
and let $A\subseteq S$. Then $d_t(A)=\sup\{\alpha\in[0,1]:\big(\forall F\in\pf(S)\big)
(\exists s\in S)(|F\cap As^{-1}|\geq\alpha\cdot|F|)\}$. \end{definition}

The proof of the next lemma is based on the proof of \cite[Theorem 3.2]{BG}.
This lemma  plays an important role in our paper.

\begin{lemma}\label{lemRnotemp} Let $(S,\cdot)$ be a left amenable semigroup,
let $A\subseteq S$, let $\mu$ be a left invariant mean on $S$ such that
$\mu(\cchi_A)>0$. Assume that $F\in\pf(S)$ and $0<\eta<\delta\leq\mu(\cchi_A)$
and let $R=\{s\in S:|F\cap As^{-1}|\geq\eta\cdot|F|\}$. Then
$\mu(\cchi_R)\geq\displaystyle\frac{\delta-\eta}{1-\eta}$. In particular,
$R\neq\emp$.\end{lemma}

\begin{proof} Note that for $x\in S$, $\cchi_{x^{-1}A}=\cchi_A\circ\lambda_x$ so
$\mu(\cchi_{x^{-1}A})=\mu(\cchi_A\circ\lambda_x)=\mu(\cchi_A)$.
Define $g:S\to[0,1]$ by $g(t)=\displaystyle\frac{|F\cap At^{-1}|}{|F|}$.
Then for $t\in S$, $g(t)=\frac{1}{|F|}\sum_{x\in F}\cchi_{At^{-1}}(x)=
\frac{1}{|F|}\sum_{x\in F}\cchi_{x^{-1}A}(t)$ so
$g=\frac{1}{|F|}\sum_{x\in F}\cchi_{x^{-1}A}$. Thus
$\mu(g)=\frac{1}{|F|}\sum_{x\in F}\mu(\cchi_{x^{-1}A})=
\frac{1}{|F|}\sum_{x\in F}\mu(\cchi_{A})=\mu(\cchi_A)$.

Since $\mu$ is additive, $\mu(\cchi_A)=\mu(g)\leq \mu(g\cchi_R)+\mu(g\cchi_{S\setminus R})$.
Since $g\cchi_R\leq \cchi_R$, $\mu(g\cchi_R)\leq \mu(\cchi_R)$.
For $t\in S\setminus R$, $|F\cap At^{-1}|<\eta\cdot|F|$ so 
$g(t)<\eta$. Also, $\mu(\cchi_{S\setminus R})=1-\mu(\cchi_R)$
so $\mu(\cchi_A)\leq\mu(g\cchi_R)+\mu(g\cchi_{S\setminus R})\leq\mu(\cchi_R)+\eta\mu(\cchi_{S\setminus R})=
\mu(\cchi_R)+\eta\big(1-\mu(\cchi_R)\big)$.
Therefore $\mu(\cchi_A)-\eta\leq \mu(\cchi_R)\cdot(1-\eta)$ so
$\mu(\cchi_R)\geq  \displaystyle\frac{\mu(\cchi_A)-\eta}{1-\eta}\geq\displaystyle\frac{\delta-\eta}{1-\eta}$.
\end{proof}

\begin{corollary}\label{dtgeqdstar} Let $(S,\cdot)$ be a left amenable semigroup. For every subset $A$ of $S$, $d_t(A) \geq d^*(A)$.
\end{corollary}

\begin{proof} Suppose $d_t(A)<d^*(A)$, pick $\eta$ such that $d_t(A)<\eta<d^*(A)$,
and pick $\mu\in LIM(S)$ such that $\delta=\mu(\cchi_A)>\eta$. Since $d_t(A)<\eta$
pick $F\in\pf(S)$ such that for all $s\in S$, $|F\cap As^{-1}|<\eta\cdot|F|$.
But then the set $R$ in Lemma \ref{lemRnotemp} is empty, a contradiction.
\end{proof}

Combining the conclusions of Theorem \ref{thmneteq} and Corollary \ref{dtgeqdstar},
we have shown that in any semigroup $(S,\cdot)$ satisfying SFC, for all $A \subseteq S$,
$$d(A) \leq d^*(A) \leq d_t(A).$$
Bergelson and Glasscock \cite[Theorem 3.5]{BG} showed that these three quantities are equal
if $S$ satisfies SFC and is either left cancellative or commutative. We shall show in Theorem \ref{d=dstar} in Section
\ref{sec_density_and_quotients} that the same conclusion holds for any semigroup satisfying SFC. 

It is easy to see that in any left amenable semigroup, $d^*$ is both left
invariant and subadditive.  Therefore, as a consequence of Theorem \ref{d=dstar},
if $S$ satisfies SFC, then both $d$ and $d_t$ are left invariant and subadditive. 
 While $d_t$ is defined on the subsets of every semigroup,
there are semigroups in which
$d_t$ is neither subadditive nor left invariant. 
For example, let $S$ be the free semigroup on two generators $a$ and $b$,
let $A$ be the set of elements of $S$ whose first letter is $a$, and let $B$ be the set of elements of 
$S$ whose first letter is $b$. Then $d_t(A)=d_t(B)=0$, because $\{a\} \cap Bs^{-1}=\emp$ for every $s\in S$. However,
$d_t(A\cup B)=1$. Furthermore, $d_t(a^{-1}A)=1$, because $a^{-1}A=S$.

\begin{question} Let $(S,\cdot)$ be a semigroup.
\begin{itemize}
\item[(1)] If $S$ is left amenable and $A\subseteq S$, must $d_t(A)=d^*(A)$?
\item[(2)] Is $d_t(x^{-1}A)\geq d_t(A)$ for every $A\subseteq S$ and every $x\in S$?
\end{itemize} 
\end{question}

If $S$ is any semigroup, $A\subseteq S$, and $x$ is a left cancelable element of $S$, then it is
easy to show that $d_t(x^{-1}A )\geq d_t(A)$.
We do not know of any example for which $d_t(x^{-1}A)<d_t(A)$.

\section{The ideal of ultrafilters whose members have positive density}

Given a discrete semigroup $(S,\cdot)$, the Stone-\v Cech compactification, $\beta S$
of $S$, is the set of ultrafilters on $S$. We identify a point $x\in S$ with
the principal ultrafilter $\{A\subseteq S:x\in A\}$. The topology on 
$\beta S$ has a basis consisting of the open and closed subsets $\{\overline A:A\subseteq S\}$,
where $\overline A=\{p\in\beta S:A\in p\}$. The notation is justified by the fact that $\overline A=\cl_{\beta S}(A)$.

If $X$ is a compact Hausdorff space and $f:S\to X$, we denote by $\widetilde f$ the 
continuous extension of $f$ taking $\beta S$ to $X$.
The operation on $S$ extends to $\beta S$, making $(\beta S,\cdot)$ a right topological
semigroup with $S$ contained in the topological center of $\beta S$. That is, 
for each $p\in\beta S$, the function $\rho_p:\beta S\to \beta S$ defined by
$\rho_p(q)=qp$ is continuous and for each $x\in S$, the function $\lambda_x:\beta S\to \beta S$ defined by
$\lambda_x(q)=xq$ is continuous. As does any compact Hausdorff right topological 
semigroup, $\beta S$ has a smallest two sided ideal, denoted $K(\beta S)$,
which is the union of all minimal left ideals and is also the union of all minimal right ideals.
Any two minimal left ideals are isomorphic, any two minimal right ideals are 
isomorphic, and the intersection of any minimal right ideal and any minimal left
ideal is a group.
For basic information about the algebraic structure
of $\beta S$ see \cite[Part I]{HS}.

\begin{definition}\label{defDelta} Let $(S,\cdot)$ be a semigroup satisfying 
SFC. Then \\
$\Delta(S)=\{p\in\beta S:(\forall A\in p)(d(A)>0)\}$. \end{definition}

\begin{definition}\label{defDeltastar} Let $(S,\cdot)$ be a left amenable semigroup. Then\\
$\Delta^*(S)=\{p\in\beta S:(\forall A\in p)(d^*(A)>0)\}$. \end{definition}

For any semigroup $(S,\cdot)$ that satisfies SFC, $\Delta(S)$ is a left ideal of $\beta S$. It has been an open question 
for decades whether $\Delta(S)$ is a right ideal of $\beta S$. We will show at the 
conclusion of this section that for any left amenable semigroup $S$, $\Delta^*(S)$ is an ideal of $\beta S$.
When we have shown that $d=d^*$ for semigroups satisfying SFC, the following lemma 
will be an immediate consequence. But we need it earlier, and it follows quickly from
a recent result. For the purposes of this paper we will take as the definition of 
{\it piecewise syndetic\/} the fact from \cite[Corollary 4.41]{HS} that a subset
$A$ of $S$ is piecewise syndetic if and only if $\cl A\cap K(\beta S)\neq\emp$.

\begin{lemma}\label{lemKinDelta} If $(S,\cdot)$ is a semigroup satisfying SFC,
$\cl K(\beta S)\subseteq \Delta(S)$. \end{lemma}

\begin{proof} Let $p\in \cl K(\beta S)$ and let $A\in p$. Then 
$A$ is piecewise syndetic, so by \cite[Corollary 3.5(2)]{HSc}, $d(A)>0$.
\end{proof}

It is the first point in Lemma \ref{lemma_algebraic_properties_of_ultrafilter_shift} 
that justifies the notation ``$A q^{-1}$'' in the following definition.

\begin{definition}
Let $(S,\cdot)$ be a semigroup, $A \subseteq S$, and $q \in \beta S$.  We define
$$A q^{-1} = \big\{ s \in S : s^{-1} A \in q \big\}.$$
\end{definition}

\begin{lemma} \label{lemma_algebraic_properties_of_ultrafilter_shift}
Let $(S,\cdot)$ be a semigroup.  For all $A, B \subseteq S$, $t \in S$, and $p, q \in \beta S$,
\begin{enumerate}
    \item $A \in pq$ if and only if $A q^{-1} \in p$.
    \item $t^{-1} (A q^{-1}) = (t^{-1} A) q^{-1}$.
    \item $(A \cup B)q^{-1} = Aq^{-1} \cup B q^{-1}$.
    \item $(A \cap B)q^{-1} = Aq^{-1} \cap B q^{-1}$. Therefore, if the sets $A$ and $B$ are disjoint, then the sets $Aq^{-1}$ and $Bq^{-1}$ are disjoint.
\end{enumerate}
\end{lemma}

\begin{proof} These are all routine computations using 
the fact from \cite[Theorem 4.12]{HS} that for any $p$ and $q\in\beta S$ and any
$A\subseteq S$, $A\in pq$ if and only if\\  $\{s\in S:s^{-1}A\in q\}\in p$.
\end{proof}

\begin{lemma}\label{lemma_ultrafilter_mean_shift}
Let $(S,\cdot)$ be a semigroup, $q \in \beta S$, and $\mu \in MN(S)$.  Define $\nu: \mathcal{P}(S) \to \ber$ by putting $\nu(A) = \mu(Aq^{-1})$ for every $A \subseteq S$.  Then the set function $\nu$ extends uniquely to a mean on $S$.  If $\mu$ is left invariant, then $\nu$ extends uniquely to a left invariant mean on $S$.
\end{lemma}

\begin{proof}
The set function $\nu$ is real valued and non-negative. It follows from Lemma \ref{lemma_algebraic_properties_of_ultrafilter_shift} and the finite additivity of $\mu$ that $\nu$ is finitely additive.  Moreover, since $S q^{-1} = S$, we have that $\nu (S) = 1$.  It follows from Lemma \ref{content} that $\nu$ extends uniquely to a mean on $S$.  Abusing notation slightly, we denote this mean by $\nu$.

We wish to show that if $\mu$ is left invariant, then $\nu$ is left invariant.  Let $A \subseteq S$ and $s \in S$.  We see by Lemma \ref{lemma_algebraic_properties_of_ultrafilter_shift} that
$$\nu(s^{-1} A) = \mu\big( (s^{-1} A) q^{-1} \big) = \mu \big( s^{-1} (A q^{-1}) \big) = \mu (A q^{-1}) = \nu (A).$$
It follows by Lemma \ref{s^{-1}A} that $\nu$ is left invariant.
\end{proof}

 We present now one of the major results of this paper, namely that whenever $S$ is a left amenable
semigroup, $\Delta^*(S)$ is an ideal of $\beta S$. Part of that conclusion is very easy.

\begin{lemma}\label{Delta*left} Let $(S,\cdot)$ be a left amenable semigroup. 
Then $\Delta^*(S)$ is a left ideal of $\beta S$.
\end{lemma}

\begin{proof} Let $p\in\Delta^*(S)$, let $q\in\beta S$, and let $A\in qp$.
Then $\{s\in S:s^{-1}A\in p\}\in q$ so pick $s\in S$ such that
$s^{-1}A\in p$. Then $d^*(s^{-1}A)>0$  so pick $\eta\in LIM(S)$ such that
$\eta(\cchi_{s^{-1}A})>0$. Since $\cchi_A\circ\lambda_s=\cchi_{s^{-1}A}$,
$\eta(\cchi_A)=\eta(\cchi_{s^{-1}A})$ so $d^*(A)>0$.\end{proof}

\begin{theorem}\label{Delta*ideal}  Let $(S,\cdot)$ be a left amenable semigroup. 
Then $\Delta^*(S)$ is a two sided ideal of $\beta S$.\end{theorem}

\begin{proof}  Let $p\in \Delta^*(S)$, let $q\in \beta S$, and let $A\in pq$.
 We shall show that $d^*(A)>0$. This will establish that
$pq\in \Delta^*(S)$ and hence that $\Delta^*(S)$ is a right ideal of $\beta(S)$.
It will follow then from Lemma \ref{Delta*left} that $\Delta^*(S)$ is a two-sided ideal of $\beta(S)$.

Let $P=Aq^{-1}$. Then $P\in p$ so $d^*(P)>0$, and we can 
choose $\mu\in LIM(S)$ for which $\mu(P)>0$. 

We define $\nu$ on ${\mathcal P}(S)$ by  
$\nu(B)=\mu(Bq^{-1})$ for every 
$B\subseteq S$. By Lemma \ref{lemma_ultrafilter_mean_shift}, $\nu$ extends uniquely to a 
member of $LIM(S)$. Now $P=Aq^{-1}$ so
$\nu(A)=\mu(Aq^{-1})=\mu(P)>0$. Therefore, $d^*(A) > 0$. \end{proof}

  We will need to use the following notion of size for a semigroup.
\begin{definition} \label{defthick} Let $(S,\cdot)$ be a semigroup. A set 
$A\subseteq S$ is {\it thick\/} if and only if for each $F\in\pf(S)$, there 
is some $x\in S$ such that $Fx\subseteq A$.
\end{definition}

The following theorem shows that thickness is characterized by having density 1 with respect to
each of the three notions of density considered in this paper.

\begin{theorem}\label{thick} Let $(S,\cdot)$ be a semigroup, and let $A$ be a subset of $S$.
\begin{itemize}
\item[(1)] $d_t(A)=1$ if and only if $A$ is thick.
\item[(2)] If $S$ is left amenable, $d^*(A)=1$ if and only if $A$ is thick.
\item[(3)] If $A$ satisfies  SFC, $d(A)=1$ if and only if $A$ is thick.
\end{itemize}
\end{theorem}

\begin{proof} (1) Assume that $A$ is thick. Then, for every 
finite subset $F$ of $S$, there exists $s\in S$ such that $Fs\subseteq A$
so that $|F\cap As^{-1}|=|F|$.

If $A$ is not thick, there is a finite subset $F$ of $S$ such that, 
for every $s\in S$, $Fs\not\subseteq A$. So, for every $s\in F$,  $|F \cap As^{-1}|\leq  |F|-1$
and $d_t(A)\leq \frac{|F|-1}{|F|}<1$.

(2) If $d^*(A)=1$, then by Corollary \ref{dtgeqdstar}, $d_t(A)=1$, so (1) applies.

Now assume that $A$ is thick. The family $\mathcal{F} = \{s^{-1} A : s \in S\}$ 
has the finite intersection property and 
hence is contained in some ultrafilter $q \in \beta S$.  By the 
choice of $q$, we have that $Aq^{-1} = S$. Pick $\mu\in LIM(S)$, and
define $\nu:{\mathcal P}(S)\to\ber$ by $\nu(B)=\mu(Bq^{-1})$ for all $B \subseteq S$.
By Lemma \ref{lemma_ultrafilter_mean_shift}, $\nu$ extends uniquely to a 
member of $LIM(S)$ and $\nu(A)=\mu(S)=1$, so $d^*(A)=1$.

(3) This statement was proved in \cite[Theorems 2.4 and 3.4]{HSc}. It is also a 
consequence of the fact that $d(A)=d^*(A)=d_t(A)$ for every subset $A$ of $S$, 
which we shall prove in Theorem \ref{d=dstar}. \end{proof}

\section{Density and quotients}
\label{sec_density_and_quotients}

In this section, we investigate quotients of left amenable semigroups, paying special
attention to the {\it left cancellative quotient\/} produced by 
Lemma \ref{lemquotient}. And we establish that $d=d^*$ in any semigroup satisfying
SFC.

\begin{definition}\label{defnuh} Let $(S,\cdot)$ be a left amenable semigroup, and
let $h$ be a homomorphism from $S$ onto a semigroup $(T,\cdot)$. For $\nu\in l_\infty(S)^*$, define
$\nu_h\in l_\infty(T)^*$ by $\nu_h(g)=\nu(g\circ h)$ for each
$g\in l_\infty(T)$.
\end{definition}

\begin{lemma}\label{lemnusubh}  Let $(S,\cdot)$ be a left amenable semigroup,
let $h$ be a homomorphism from $S$ onto a semigroup $(T,\cdot)$, and let $\nu\in LIM(S)$. 
Then $\nu_h\in LIM(T)$.\end{lemma}

\begin{proof}  Given $g\in l_\infty(T)$, $||g\circ h||_\infty=||g||_\infty$,
so $\nu_h$ is a mean. To see that $\nu_h$ is left invariant, let $g\in l_\infty(T)$
and let $x\in T$. Pick $s\in S$ such that $h(s)=x$.
Then $g\circ \lambda_x\circ h=g\circ h\circ\lambda_s$ so
$\nu_h(g\circ\lambda_x)=\nu(g\circ \lambda_x\circ h)=\nu(g\circ h\circ\lambda_s)
=\nu(g\circ h)=\nu_h(g)$.
\end{proof}

\begin{theorem}\label{thmnusurj} Let $(S,\cdot)$ be a left amenable semigroup,
let $h$ be a homomorphism from $S$ onto a semigroup $(T,\cdot)$, and let $\mu\in LIM(T)$. There exists 
$\nu\in LIM(S)$ such that $\nu_h=\mu$. \end{theorem}

\begin{proof}  Let $E=\{f\circ h:f\in l_\infty(T)\}$. Then $E$ is a linear 
subspace of $l_\infty(S)$. Define
$\eta:E\to\ber$ by $\eta(f\circ h)=\mu(f)$, noting that $\eta$ is well
defined.  Note also that if $g\in E$ and $g(s)\geq 0$ for all $s\in S$,
then $\eta(g)\geq 0$. Since $\cchi_S=\cchi_T\circ h$ we have that 
$\eta(\cchi_S)=\mu(\cchi_T)=1$ so $||\eta||=1$. 

We need to produce $\nu\in LIM(S)$ which agrees with $\eta$ on $E$.

We claim that, for every $g\in E$ and every $s\in S$, 
$g\circ \lambda_s\in E$ and $\eta(g\circ \lambda_s)=\eta(g)$. To see this,
let $t=h(s)$ and let $g=f\circ h$, where $f\in l_{\infty}(T)$. Observe that 
$(f\circ h\circ \lambda_s)(x)=f\big(h(sx)\big)=(f\circ \lambda_t\circ h)(x)$ for every $x\in S$. 
So $g\circ \lambda_s=f\circ \lambda_t\circ h$. Then $g\circ \lambda_s\in E$ and
$\eta(g\circ \lambda_s)=\mu(f\circ \lambda_t)=\mu(f)=\eta(g)$. 

By \cite[Theorem B.14]{HR} (a version of the Hahn-Banach Theorem), there 
is an extension $\widetilde\eta$ of $\eta$ to $l_\infty(S)$ with $||\widetilde\eta||=1$.

Let $X=\{\rho\in l_\infty(S)^*:||\rho||=1$ and $(\forall g\in E)\big(\rho(g)=\eta(g)\big)\}$.
Then $\widetilde\eta\in X$ so $X\neq\emp$.
We claim that if $\rho\in X$, $\rho\geq 0$ so that the members of $X$ are all means. 
To see this, suppose that $\rho(g)<0$ for some $g\geq 0$ in $l_{\infty}(S)$. We may suppose that
$||g||\leq 1$. Then $|| \cchi_S-g||\leq 1$ and $\rho(\cchi_S-g)>1$, 
contradicting the assumption that $||\rho||=1$.

For $s\in S$ and $\rho\in X$, we define $s*\rho$ in $X$ by $s*\rho(g)=\rho(g\circ\lambda_s)$
for each $g\in l_\infty(S)$.  We want to apply Day's Fixed Point Theorem \cite[Theorem 1.14]{P}.
The conclusion of that theorem is that there is some $\nu\in X$ such that
for each $s\in S$, $s*\nu=\nu$. Then for each $g\in l_\infty(S)$ and each
$s\in S$, $\nu(g\circ\lambda_s)=\nu(g)$, so that $\nu$ is a left invariant mean on $S$.
Since $\nu\in X$, for each $f\in l_\infty(T)$, $\nu_h(f)=\nu(f\circ h)=\eta(f\circ h)=\mu(f)$, as 
required.

To apply Day's Fixed Point Theorem, we need to show
\begin{itemize}
\item[(1)] for $s\in S$ and $\rho\in X$, $s*\rho\in X$;
\item[(2)] for $s,t\in S$ and $\rho\in X$, $s*(t*\rho)=(st)*\rho$;
\item[(3)] $X$ is compact in $\bigtimes_{f\in l_\infty(S)}[-||f||_\infty,||f||_\infty]$;
\item[(4)] $X$ is convex; and 
\item[(5)] for $\alpha\in [0,1]$, $\rho_1$ and $\rho_2$ in $X$, and $s\in S$,
$s*(\alpha\rho_1+(1-\alpha)\rho_2)=\alpha(s*\rho_1)+(1-\alpha)(s*\rho_2)$.
\end{itemize}

For (1), we need that $||s*\rho||=1$ and for each $g\in E$, $(s*\rho)(g)=\eta(g)$.
Since $||s*\rho||\leq||\rho||=1$ and $s*\rho(\cchi_S)=\rho(\cchi_S)=1$ we have that 
$||s*\rho||=1$. Given $g\in E$, since $g\circ\lambda_s\in E$, $(s*\rho)(g)=\rho(g\circ\lambda_s)=
\eta(g\circ\lambda_s)=\eta(g)$, whereby $s*\rho \in X$.

For (2), let $g\in l_\infty(S)$. Then $s*(t*\rho)(g)=(t*\rho)(g\circ\lambda_s)
=\rho(g\circ\lambda_s\circ\lambda_t)=\rho(g\circ\lambda_{st})=(st)*\rho(g)$.

It is a routine exercise to establish (3). To verify (4), let
$n\in\ben$,  let $\langle c_i\rangle_{i=1}^n$ be elements of $[0,1]$,
such that $\sum_{i=1}^nc_i=1$, and let $\langle \rho_i\rangle_{i=1}^n$ be elements of $X$.
Given $g\in E$, $(\sum_{i=1}^n c_i \rho_i)(g)=
(\sum_{i=1}^n c_i \rho_i(g))=(\sum_{i=1}^n c_i \eta(g))
=(\sum_{i=1}^n c_i) \eta(g)=\eta(g)$. And, since the norm is additive on the
set of means, 
$||\sum_{i=1}^n c_i \rho_i||=\sum_{i=1}^n c_i ||\rho_i||=1$.

Finally, the verification of (5) is a routine evaluation.
\end{proof}

\begin{theorem}\label{thmdB} Let $(S,\cdot)$ be a left amenable semigroup,
let $h$ be a homomorphism from $S$ onto a semigroup $(T,\cdot)$, and let $B\subseteq T$. Then $d^*(B)=d^*(h^{-1}[B])$.
\end{theorem}

\begin{proof} Suppose first that we have some $\delta$ such that
$d^*(B)>\delta>d^*(h^{-1}[B])$. By Theorem \ref{thmnusurj}, pick $\mu\in LIM(T)$ such that
$\mu(\cchi_B)>\delta$ and pick $\nu\in LIM(S)$ such that
$\nu_h=\mu$. Then $\delta< \mu(\cchi_B)=\nu(\cchi_B\circ h)=\nu(\cchi_{h^{-1}[B]})\leq d^*(h^{-1}[B])$,
a contradiction. 

Now suppose we have some $\delta$ such that $d^*(h^{-1}[B])>\delta>d^*(B)$
and pick $\nu\in LIM(S)$ such that
$\nu(\cchi_{h^{-1}[B]})>\delta$. Then $d^*(B)\geq \nu_h(\cchi_B)>\delta$, a contradiction.
\end{proof}

\begin{theorem}\label{thmdhA} Let $(S,\cdot)$ be a left amenable semigroup,
let $h$ be a homomorphism from $S$ onto a semigroup $(T,\cdot)$, and let $A\subseteq S$.
Then $d^*(h[A])\geq d^*(A)$.
\end{theorem}

\begin{proof} Since $A\subseteq h^{-1}\big[h[A]\big]$, by Theorem \ref{thmdB},
$d^*(A)\leq d^*\big(h^{-1}\big[h[A]\big]\big)=\\
d^*(h[A])$.  \end{proof}

Denote by $\widetilde h: \beta S \to \beta T$ the continuous extension of $h$. Note that by 
\cite[Corollary 4.22]{HS}, $\widetilde h$ is a homomorphism of $\beta S$ onto $\beta T$.

\begin{theorem}\label{thmdeltastar} Let $(S,\cdot)$ be a left amenable semigroup, and
let $h$ be a homomorphism from $S$ onto a semigroup $(T,\cdot)$.  Then $\widetilde h[\Delta^*(S)]=\Delta^*(T)$.
\end{theorem}

\begin{proof} 
To see that $\widetilde h[\Delta^*(S)]\subseteq\Delta^*(T)$, let $q\in \widetilde h[\Delta^*(S)]$
and pick $p\in \Delta^*(S)$ such that $\widetilde h(p)=q$. Suppose $q\notin \Delta^*(T)$ and
pick $B\in q$ such that $d^*(B)=0$. Then $h^{-1}[B]\in p$, and by Theorem \ref{thmdB}, $d^*(h^{-1}[B])=0$, a contradiction.

To see that $\Delta^*(T)\subseteq \widetilde h[\Delta^*(S)]$, let $q\in \Delta^*(T)$.
Then, by Theorem \ref{thmdB}, for each $B\in q$, $d^*(h^{-1}[B])>0$. 
Let ${\mathcal A}=\{h^{-1}[B]:B\in q\}$ and let ${\mathcal R}=\{A\subseteq S:d^*(A)>0\}$.
Then by \cite[Theorem 3.11]{HS}, there exists $p\in \beta S$ such that
${\mathcal A}\subseteq p\subseteq{\mathcal R}$. So $p\in\Delta^*(S)$
and $\widetilde h(p)=q$.
\end{proof}

It is a consequence of Lemma \ref{lemlambdaright} that if $(S,\cdot)$ is a left amenable semigroup, then the intersection
of finitely many right ideals of $S$ is nonempty and, hence, is a right ideal.

\begin{lemma}\label{defR} Let $(S,\cdot)$ be a left amenable semigroup, let $n\in \ben$, 
and let\break $a_1,a_2,\ldots,
a_n$ and $b_1,b_2,\ldots,b_n$ be elements of $S$ 
with the property that for all $i\in\{1,2,\ldots,n\}$, there exists $x \in S$ such that $a_i x = b_i x$. 
There is a right ideal $R$ of $S$ such that $a_iu=b_iu$ for every
$u\in R$ and every $i\in \{1,2,\ldots,n\}$. \end{lemma}

\begin{proof} Let $R_i=\{u\in S:a_iu=b_iu\}$. Then each $R_i$ is a right ideal of $S$. Since $(S,\cdot)$ is left amenable, $R=\bigcap_{i=1}^n R_i$ is a right ideal. \end{proof}

\begin{lemma}\label{lemquotient} Let $(S,\cdot)$ be a semigroup satisfying SFC
and define a relation $\sim$ on $S$ by, for $a,b\in S$, $a\sim b$
if and only if there exists $x\in S$ such that $ax=bx$. Then $\sim$ 
is an equivalence relation on $S$ and the quotient $T=S/\hskip -2 pt \sim\hskip 2 pt$ is a
cancellative semigroup which satisfies SFC. \end{lemma}

\begin{proof} This is \cite[Lemma 3.2]{HSc}. Its proof was  based on the proofs of
\cite[Lemma 2 and Remark 3]{G} and \cite[Theorem 2.2]{K}. \end{proof}

Throughout the rest of this section we will assume that
$(S,\cdot)$ is a semigroup satisfying SFC, $\sim$ is the equivalence relation of Lemma \ref{lemquotient}, 
$(T,\cdot)$ is the cancellative quotient of $S$,
$h:S\to T$ is the projection map, and $\widetilde h:\beta S\to \beta T$ is the continuous 
extension of $h$.

\begin{lemma}\label{lemxpyp} Let $p\in\Delta(S)$ and let $x,y\in\beta S$. 
If $\widetilde h(x)=\widetilde h(y)$, then $xp=yp$.
\end{lemma} 

\begin{proof} Assume that $\widetilde h(x)=\widetilde h(y)=q$ and
suppose that $xp\neq yp$. Pick $A\in xp\setminus yp$. Since $\rho_p$
is continuous, pick $X\in x$ and $Y\in y$ such that 
$\rho_p[\,\overline X\,]\subseteq \overline A$ and
$\rho_p[\,\overline Y\,]\subseteq \overline{S\setminus A}$.
Then $q\in\widetilde h[\,\overline X\,]=\cl h[X]$ and
$q\in\widetilde h[\,\overline Y\,]=\cl h[Y]$. Since
$\beta S$ is extremally disconnected, $h[X]\cap h[Y]\neq\emp$,
so pick $a\in X$ and $b\in Y$ such that $h(a)=h(b)$. By Lemma \ref{lemapbp},
$ap=bp$. This is a contradiction since $ap\in\overline A$ and
$bp\in\overline{S\setminus A}$. \end{proof} 

\begin{theorem}\label{thmhiso} If there is an element of
$\Delta(S)$ which is right cancelable in $\beta S$,
then $\widetilde h:\beta S\to\beta T$ is an isomorphism.
\end{theorem}

\begin{proof} Pick $p\in\Delta(S)$ such that $p$
is right cancelable in $\beta S$. Since $\widetilde h$
is a surjective homomorphism,  it sufices to show that $\widetilde h$ is injective. Let $x,y\in\beta S$
and assume that $\widetilde h(x)=\widetilde h(y)$. Then by Lemma \ref{lemxpyp},
$xp=yp$ so $x=y$.\end{proof}

\begin{corollary}\label{corhiso} If there is an element of
$\cl K(\beta S)$ which is right cancelable in $\beta S$,
then $\widetilde h:\beta S\to\beta T$ is an isomorphism.
\end{corollary}

\begin{proof} By Lemma \ref{lemKinDelta}, $\cl K(\beta S)\subseteq \Delta(S)$.
\end{proof}

\begin{corollary}\label{corScanc} If there is an element of
$\Delta(S)$ which is right cancelable in $\beta S$, then
$S$ is cancellative.\end{corollary}

\begin{proof} By Theorem \ref{thmhiso}, $S$ and $T$ are
isomorphic, and by Lemma \ref{lemquotient}, $T$ is cancellative.\end{proof}

If $S$ is cancellative and countable, $\cl K(\beta S)$ does contain right cancelable elements (\cite[Corollary 8.26]{HS}). It follows
from Corollary \ref{corScanc} that, for a countable semigroup $S$ which satisfies SFC, the existence of right cancelable elements of $\beta S$ in $\cl K(\beta S)$ is equivalent to $S$ being cancellative.

\begin{theorem}\label{thmisoleft} The minimal left ideals of
$\beta S$ and $\beta T$ are isomorphic.\end{theorem}

\begin{proof} Let $L$ be a minimal left ideal of $\beta S$. By \cite[Exercise 1.7.3]{HS},
$\widetilde h[L]$ is a minimal left ideal of $\beta T$. We will show that the
restriction of $\widetilde h$ to $L$ is an isomorphism. So let
$a,b\in L$, and assume that $\widetilde h(a)=\widetilde h(b)$.
Pick an idempotent $p$ in $L$. By Lemma \ref{lemxpyp}, $ap=bp$ so, 
since $p$ is a right identity for $L$ by \cite[Lemma 1.30]{HS},
$a=b$. \end{proof}

\begin{theorem}\label{Tfinite} Assume that $T$ is finite. Then:

\begin{itemize}
\item[(1)] $T$ is a finite group.
\item[(2)] $T$ is the unique minimal left ideal of $\beta T$.
\item[(3)] The minimal left ideals of $\beta S$ are isomorphic to $T$.
\item[(4)] $\beta S$ has a unique minimal right ideal.
\item[(5)] $T=K(\beta T)=\Delta(T)$.
\end{itemize}
\end{theorem}

\begin{proof}
(1) By Lemma \ref{lemquotient}, $T$ is cancellative. As a finite
cancellative semigroup, $T$ is a group.

(2) Since $T$ has a unique idempotent, it has only one minimal left
ideal. If $e$ is the identity of $T$, then $T=Te$.

(3) This follows from (2) and Theorem \ref{thmisoleft}.

(4) Let $L$ be a minimal left ideal of $\beta S$. By (3), $L$ is
isomorphic to $T$, so it has a unique idempotent. If $R_1$ and $R_2$
were distinct (hence disjoint) minimal right ideals of $\beta S$, 
both $R_1\cap L$ and $R_2\cap L$, being groups,  would have idempotents.

(5) By (2), Lemma \ref{lemKinDelta}, and the fact that $K(\beta T)$ is the 
union of the minimal left ideals of $\beta T$, 
$T=K(\beta T)\subseteq \Delta(T)\subseteq \beta T=T$.
\end{proof}

In this part of this section, we are assuming that $S$ satisfies SFC. 
Since the next result is probably the most important result of the paper,
we recall in its statement that SFC is the only assumption needed.

\begin{theorem}\label{d=dstar} Let $(S,\cdot)$ be a semigroup which satisfies SFC. 
Then, for every subset $A$ of $S$, $d(A)=d^*(A)=d_t(A)$.
\end{theorem}

\begin{proof} By Theorem \ref{thmneteq} and Corollary \ref{dtgeqdstar} we have that
$d(A)\leq d^*(A)\leq d_t(A)$ so it suffices to show that $d_t(A)\leq d(A)$. 
To see this we will show that if $\eta<d_t(A)$, then $d(A)\geq \eta$.
So let $\eta<d_t(A)$ be given.

To see that $d(A)\geq \eta$, let $H\in\pf(S)$ and $\epsilon>0$ be given
and let ${\mathcal F}=\{K\in\pf(S):(\forall s\in H)(|K\setminus sK|<\epsilon\cdot|K|)\}$.
We shall show that there exists $K\in{\mathcal F}$ such that $|K\cap A|\geq\eta\cdot |K|$.

Let $g:T\to S$ be a choice function for $\sim$ and note that $h\big(g(t)\big)=t$ for every 
$t\in T$.  Observe that $g$ is injective and $h$ is injective on $g[T]$.
We claim that for each $s\in S$, $\rho_s$ is injective on $g[T]$. To see this, let
$a,b\in T$ and assume that $g(a)s=g(b)s$. Then 
$g(a)\sim g(b)$ so $a=h\big(g(a)\big)=h\big(g(b)\big)=b$.

 Now $h[H]\in\pf(T)$ and by Lemma \ref{lemquotient}, $T$ satisfies SFC, so
pick $F\in\pf(T)$ such that for all $s\in H$, $|F\setminus h(s)F|<\epsilon\cdot |F|$.
 Now $y\in F\cap h(s)F$ if and only if $y\in F$ and
$y=h(s)y'$ for some $y'\in F$. The equation 
$y=h(s)y'$ implies that $h\big(sg(y')\big)=h(s)h\big(g(y')\big)=h(s)y'=y=h\big(g(y)\big)$
so, by Lemma \ref{defR}, there 
is a right ideal $R$ of $S$ such that $g(y)u=sg(y')u$ for every $u \in R$, every $y,y'\in F$, and 
every $s\in H$ for which $y=h(s)y'$. 

For each $u\in R$, let $G_u=g[F]u$ and note that 
$|G_u|=|g[F]u|=|g[F]|=|F|$, because $\rho_u$ is injective on $g[F]$ and $g$ is injective.
We claim that for each $u\in R$, $G_u\in {\mathcal F}$. To see this, let
$u\in R$ and let $s\in H$. It suffices to show that $G_u\setminus sG_u
\subseteq \{g(y)u:y\in F\setminus h(s)F\}$, for then
$|G_u\setminus sG_u|\leq |F\setminus h(s)F|<\epsilon\cdot |F|=\epsilon\cdot |G_u|$.
So let $x\in G_u\setminus sG_u$ and pick $y\in F$ such that $x=g(y)u$. If we had
$y\in h(s)F$, there would be some $y'\in F$ such that $y=h(s)y'$ so
$x=g(y)u=sg(y')u$ and thus $x\in sG_u$.

Choose $u\in R$.  Since $d_t(A)>\eta$, we may pick $x\in S$ such that $|G_u\cap Ax^{-1}|\geq\eta\cdot|G_u|$.
 If $z\in G_u\cap Ax^{-1}$, then $zx\in G_ux\cap A$.  So $(G_u\cap Ax^{-1})x \subseteq G_ux\cap A$. 
Now $\rho_x$ is injective on $G_u$, because $\rho_{ux}$ is injective on $g[F]$. 
So $|G_ux\cap A|\geq|G_u\cap Ax^{-1}|$ and therefore
$|G_ux\cap A|\geq\eta\cdot|G_u|=\eta\cdot|G_ux|$. Since $G_ux=G_{ux}\in {\mathcal F}$, we are done.
\end{proof}

\begin{corollary}\label{Deltaideal} If $(S,\cdot)$ is a 
semigroup which satisfies SFC, then $\Delta(S)$ is a two sided ideal in $\beta S$.
\end{corollary}

\begin{proof} Theorems \ref{Delta*ideal} and \ref{d=dstar}.\end{proof}

\begin{corollary}\label{cordbdeltab} $\widetilde h[\Delta^*(S)]=\Delta(T)=\Delta^*(T)$ and
for each $B\subseteq T$, $d(B)=d^*(h^{-1}[B])$. \end{corollary}

\begin{proof} By Theorem \ref{d=dstar}, $\Delta(S) = \Delta^*(S)$ and $\Delta(T) = \Delta^*(T)$.  The conclusions are then
immediate consequences of Theorems \ref{thmdB} and \ref{thmdeltastar}.
\end{proof}

\section{When minimal left ideals are singletons}

Hindman and Strauss have dealt with versions of this subject before. In \cite{HSx} they had a section
titled {\it Semigroups with isolated points in minimal left ideals\/}; in \cite{HSy} they
had a section titled {\it Finitely many minimal right ideals\/}; and in \cite{HSz} they
had a section titled {\it Finite minimal left ideals\/} and showed that $\beta S$ has 
finite minimal left ideals if and only if it has finitely many minimal right ideals.

In this paper, we arrive at this subject from a different direction and with
a particular interest in density, which the aforementioned works do not address. As 
we noted earlier, Bergelson and Glasscock showed in \cite[Theorem 3.5]{BG} that if
$S$ satisfies SFC and is either left cancellative or commutative, then for 
each $A\subseteq S$, $d(A)=d^*(A)$. Also, it is an old fact noted by Klawe in \cite[Corollary 2.3]{K}
that if $S$ satisfies SFC and is right cancellative, then $S$ is also left cancellative. 
At that point we were actively considering the possibility of finding a
counterexample to the assertion that $d$ and $d^*$ are always equal.
So to find a semigroup for which $d$ and $d^*$ are not equal, we needed
a semigroup which satisfies SFC and is not commutative and not left or right cancellative.
The following simple observation provides such examples.

\begin{lemma}\label{lemRTzero} Let $(S,\cdot)$ be a semigroup. If 
$S$ has a right zero, then $S$ satisfies SFC. \end{lemma}

\begin{proof} Let $z$ be a right zero in $S$. Let $H\in\pf(S)$, and
let $\epsilon>0$. Let $K=\{z\}$. Then for all $s\in H$, $sK=K$ so
$|K\setminus sK|=0<\epsilon\cdot|K|$. \end{proof}

If $S$ is the semigroup of $2\times 2$ matrices over the 
set $\omega$ of nonnegative integers, then $S$ satisfies SFC, is not commutative, and
is neither  right nor left cancellative. 

\begin{lemma}\label{lemDeltasubAbar} Let $(S,\cdot)$ be a semigroup which satisfies SFC,
let $A\subseteq S$, and assume that for each $\nu\in LIM_0(S)$, $\nu(\cchi_A)=1$.
Then $\Delta(S)\subseteq \overline A$.\end{lemma}

\begin{proof} Suppose we have some $p\in\Delta(S)\setminus\overline A$.
Then $S\setminus A\in p$ so $d(S\setminus A)>0$ so by Theorem
\ref{thmneteq} we may pick $\nu\in LIM_0(S)$ such that $\nu(\cchi_{S\setminus A})>0$.
But since $\nu(\cchi_A)=1$, one must have $\nu(\cchi_{S\setminus A})=0$, a
contradiction.
\end{proof}

\begin{lemma}\label{lemcSsubC} Let $(S,\cdot)$ be a semigroup which satisfies SFC,
let $C$ be a compact subset of $\beta S$, and assume that there exists
$x\in S$ such that $xS\subseteq C$. Then $\Delta(S)\subseteq C$.\end{lemma}

\begin{proof} Pick $x\in S$ such that $xS\subseteq C$, and let $A=xS$. By
Lemma \ref{lemDeltasubAbar}, it suffices to show that for each $\mu\in LIM_0(S)$, $\mu(\cchi_A)=1$.
In fact, for each $\mu\in LIM(S)$, $\mu(\cchi_A)=1$. Since $A$ is a right ideal of $S$, this follows
immediately from Lemma \ref{lemlambdaright}.
\end{proof}

\begin{lemma}\label{lemapbp} Let $(S,\cdot)$ be a semigroup which satisfies SFC.
Let $a,b\in S$, and assume that there exists $x\in S$ such that $ax=bx$.
Then for all $p\in\Delta(S)$, $ap=bp$.\end{lemma}

\begin{proof}  Pick $x\in S$ such that $ax=bx$. Let $C=\{p\in\beta S:ap=bp\}$.
Then $C$ is compact and $xS\subseteq C$ so by Lemma \ref{lemcSsubC}, $\Delta(S)\subseteq C$.\end{proof}

The following lemma does not need the assumption that $S$ satisfies SFC.

\begin{lemma}\label{lemsteqt} Let $(S,\cdot)$ be a semigroup. Assume that
the minimal left ideals of $\beta S$ are singletons and $p\in K(\beta S)$.
Then for every $s\in S$, $\{t\in S:st=t\}\in p$.\end{lemma}

\begin{proof} Let $s\in S$. Since $\{p\}$ is a left ideal, $sp=p$ so
by \cite[Theorem 3.35]{HS}, $\{t\in S:st=t\}\in p$.\end{proof}

Observe that examples of semigroups which satisfy statement 
(1) of Theorem \ref{impliesSFC} are abundant because this class of 
semigroups includes all semilattices and all semigroups which have a right zero.

\begin{theorem}\label{impliesSFC} For every semigroup $(S,\cdot)$, statements (1)-(6) are equivalent.
Each of statements (1)-(6) implies statements (7) and (8).
\begin{itemize}
\item[(1)] For every $a,b\in S$, there exists $x\in S$ such that $ax=bx$.
\item[(2)] For every $H\in \pf(S)$, there exists $x\in S$ such that $ax=bx$ for every $a,b\in H$.
\item[(3)] The minimal left ideals of $\beta S$ are singletons.
\item [(4)] For every $p\in K(\beta S)$ and every $q\in \beta S$, $qp=p$.
\item[(5)]  $K(\beta S)$ is a right zero semigroup, that is, all elements of $K(\beta S)$ are right zeros of $K(\beta S)$.
\item[(6)] The semigroup $S$ satisfies SFC and for every subset $A$ of $S$, $d(A)\in \{0,1\}$.
\item[(7)] Every member of $K(\beta S)$ is idempotent.
\item[(8)] The semigroup $S$ satisfies SFC and $K(\beta S)=\Delta(S)$.
\end{itemize}  
\end{theorem}

\begin{proof} To show that (1) implies (2), assume that (1) holds. We shall show that (2) holds by induction on $|H|$.
Assume that  that $n>2$ is an integer, and that (2) holds for every $H\in \pf(S)$ for which $|H| <n$.
 Choose $H\in \pf(S)$ with $|H|=n$, and choose any $a\in H$.  There exists $x\in S$ such that $bx=cx$ for every 
$b,c\in H\setminus\{a\}$. Choose any $b\in H\setminus \{a\}$, and then choose $y\in S$ such that $ay=by$. There exists $z \in S$
such that $xz=yz$. So $bxz=cxz$ for every $c\in H\setminus \{a\}$, and $ayz=byz$. If $w=xz=yz$, then $cw=dw$ for every $c,d\in H$.

We now show that (2)  implies (3). Assume that (2) holds. 
For each $H\in\pf(S)$, pick $x_H\in S$ such that for all $a$ and $b$ in $H$, $ax_H=bx_H$. The relation $\subseteq$ directs $\pf(S)$.
Let $p$ be a cluster point of the net $\langle x_H\rangle_{H\in\pf(S)}$ in $\beta S$. Then $ap=bp$ for all $a$ and $b$ in $S$,
so $|Sp|=1$ and $\beta Sp=\cl(Sp)$ so $|\beta S p|=1$. Thus $\beta S$ has 
a minimal left ideal which is a singleton. Since all minimal left ideals are
isomorphic, all minimal left ideals are singletons.

That (3) implies (4) and that (4) implies (5) are both trivial.

To see that (5) implies (1), pick $p\in K(\beta S)$. Let $a,b\in S$. Then $ap$ and $bp$ are in $K(\beta S)$
so $ap=a(pp)=(ap)p=p$ and $bp=p$, so by Lemma \ref{lemsteqt}, $\{t\in S:at=t\}\in p$ and $\{t\in S:bt=t\}\in p$.
Choose $x$ in the intersection of these two sets so that $ax = bx$.

We have shown thus far that statements (1)-(5) are equivalent.

We now show that if (2) and (3) hold, so does (6). 
So assume that (2) and (3) hold.  Let $H\in\pf(S)$, and let $\epsilon>0$.
Pick $x\in S$ such that for all $a$ and $b$ in $H\cup HH$, $ax=bx$. Pick $a\in H$, and let $K=\{ax\}$.
Then for each $s\in H$, $|K\setminus sK|=0<\epsilon\cdot |K|$. Therefore, $S$ satisfies SFC.

Let $A\subseteq S$. If $d^*(A)=0$, then $d(A)=0$.  If, on the other hand, $d^*(A)>0$, we shall show that $A$ is thick. It will then follow from
\cite[Theorem 3.4]{HSc} that $d(A)=1$ and that (6) holds. Assume that $d^*(A)>0$, and choose $\lambda\in LIM(S)$ for which $\lambda(A)>0$.
Let $F \in \pf(S)$ and $p\in K(\beta S)$. By Lemma \ref{lemsteqt}, for every $s\in S$, $\{t\in S:st=t\}\in p$. So $R=\bigcap_{s\in F}\{t\in S:st=t\}\in p$. Now
$R$ is a right ideal of $S$ and so, by Lemma \ref{lemlambdaright}, $\lambda(R)=1$,
and hence  $\lambda(S\setminus R)=0$.
Since $\lambda(A)>0$, we have $A\cap R\neq \emp$. Pick $t\in A\cap R$.
Then $Ft=\{t\}\subseteq A$ and so $A$ is thick. 

Now we will show that (6) implies (3). 
Recall that a set $A\subseteq S$ such that $\overline A\cap K(\beta S)\neq\emp$ is piecewise syndetic.
By \cite[Corollary 3.5(2)]{HSc}, if $A$ is piecewise syndetic, then $d(A)>0$, so if $A$ is a piecewise syndetic
subset of $S$, then $d(A)=1$.

Let $L$ be a minimal left ideal of $\beta S$ and pick an idempotent $p\in L$.
Then $\beta Sp=L$. We claim that $\beta Sp=\{p\}$ for which it suffices 
that $Sp=\{p\}$. Suppose instead that there is some $s\in S$ such that $sp\neq p$. 
Choose disjoint subsets $P$ and $Q$ of $S$ such that $P\in p$ and $Q\in sp$. Since $s^{-1}Q \cap P\in p$, 
$d(s^{-1}Q \cap P)=1$. It follows from Theorem \ref{thmmax} 
that there exists $\nu\in LIM(S)$ such that $\nu(s^{-1}Q \cap P)=1$. But this implies that
$\nu(Q)=\nu(s^{-1}Q)=1$ and $\nu(P)=1$, which contradicts the assumption that $P\cap Q=\emptyset$. Therefore, (3) holds.

We have shown thus far that statements (1)-(6) are equivalent. Now we will show that each of (1)-(6) implies (7) and (8).

That (5) implies (7) is trivial. To see that (3) implies (8), assume that (3) holds. Since (3) 
implies (6) we have $S$ satisfies SFC so by Lemma \ref{lemKinDelta},
$K(\beta S)\subseteq \Delta(S)$. To see that $\Delta(S)\subseteq K(\beta S)$,
let $p\in \Delta(S)$. We shall show that for all $a\in S$, $ap=p$ so that $\beta Sp=\{p\}$
and thus $\{p\}$ is a minimal left ideal. So let $a\in S$, and pick $q\in K(\beta S)$.
Then $aq\in\beta Sq=\{q\}$ so by Lemma \ref{lemsteqt}, we may pick $s\in S$ such that
$as=s$. Then $sS\subseteq \{x\in\beta S:ax=x\}$. The set $\{x\in\beta S:ax=x\}$ is compact, 
and so it contains $\Delta(S)$ by Lemma \ref{lemcSsubC} so that $ap=p$, as claimed.
 \end{proof}

 We conclude this section by characterizing sets with positive density.

\begin{theorem}\label{thmdirect} Let $(S,\cdot)$ be a semigroup for
which the mimimal left ideals of $\beta S$ are singletons. Define a
relation $\prec$ on $S$ by $s\prec t$ if and only if $st=t$. The
relation $\preceq$ is a directed set order on $S$. Given $A\subseteq S$, the
following statements are equivalent. 
\begin{itemize}
\item[(a)] $d(A)>0$.
\item[(b)] $d(A)=1$.
\item[(c)] $A$ is cofinal in $S$ with respect to $\preceq$.
\end{itemize}\end{theorem}

\begin{proof} By $s\preceq t$ we mean, of course, that $s\prec t$ or $s=t$, so
$\preceq$ is reflexive. To verify transitivity, assume $x\preceq y$ and $y\preceq z$.
If either $x=y$ or $y=z$, then trivially $x\preceq z$. So assume that
$x\prec y$ and $y\prec z$; that is $xy=y$ and $yz=z$. Then $xz=xyz=yz=z$. To complete the 
proof that $\preceq$ directs $S$, let $x,y\in S$. Pick $p\in K(\beta S)$. 
By Lemma \ref{lemsteqt}, $\{t\in S:xt=t\}\cap \{t\in S:yt=t\}\in p$ so is nonempty.

It  follows from Theorem \ref{impliesSFC} that (a) and (b) are equivalent.

To see that (c) implies (b), assume that
$A$ is cofinal in $S$. To see that $d(A)=1$, let $H\in\pf(S)$ and $\epsilon>0$ be given.
Using Lemma \ref{lemsteqt}, pick $y\in S$ such that for each $s\in H$, $s\preceq y$.
Pick $x\in A$ such that $y\preceq x$, and let $K=\{x\}$. Then
for all $s\in H$, $K=sK$, so $|K\setminus sK|=0<\epsilon\cdot |K|$ and $|K\cap A|=|K|$.

To see that (a) implies (c), we assume that $A$ is not cofinal and show that 
$d(A)=0$. By Theorem \ref{impliesSFC}, $K(\beta S)=\Delta(S)$, so 
it suffices to show that $\overline A\cap K(\beta S)=\emp$. Suppose 
instead that we have some $p\in\overline A\cap K(\beta S)$. Since
$A$ is not cofinal, we may pick some $s\in S$ such that there is 
no $t\in A$ with $s\preceq t$. But by Lemma \ref{lemsteqt},
$\{t\in S:st=t\}\in p$. Picking $t\in A\cap\{t\in S:st=t\}$
gives a contradiction.
\end{proof}

\begin{corollary}\label{corrightzero} Let $(S,\cdot)$ be a semigroup with
a right zero element, let $Z=\{z \in S : z$ is a right zero in $S\}$, and let
$A\subseteq S$. Then
$$d(A)=\left\{\begin{array}{cl}
1&\hbox{if }A\cap Z\neq\emp\\
0&\hbox{if }A\cap Z=\emp\,.
\end{array}\right.$$
\end{corollary}

\begin{proof} Let $\preceq$ be as in Theorem \ref{thmdirect}. We show that
$A$ is cofinal in $S$ if and only if $A\cap Z\neq\emp$.
If $z\in Z\cap A$, then $\{z\}$ is cofinal in $S$, so $A$ is cofinal in $S$. If 
$A$ is cofinal in $S$, pick $z\in Z$ and pick $y\in A$ such that $z\preceq y$. Then either $y=z\in Z$ or
$zy=y$ so that $y\in Z$.
\end{proof}

\section{Density is determined by an arbitrary F\o lner net}

We saw in Theorem \ref{thmneteq} that if $(S,\cdot)$ satisfies SFC and 
$A\subseteq S$, then there is a F\o lner net $\langle F_\alpha\rangle_{\alpha\in D}$ in $\pf(S)$
such that $\displaystyle d(A)=\lim_{\alpha\in D}\frac{|A\cap F_\alpha|}{|F_\alpha|}$.
We will show in Theorem \ref{theorem_density_by_arbitrary_folner_net}
that in any semigroup satisfying SFC, the 
density $d(A)$ -- and, hence, by Theorem \ref{d=dstar}, the
densities $d^*(A)$ and $d_t(A)$ -- are determined by an arbitrary F\o{}lner net. This improves on
a result of Bergelson and Glasscock \cite[Corollary 3.6]{BG}, who showed that if $(S,\cdot)$ is countable 
and right cancellative, then given any F\o lner sequence $\langle F_n\rangle_{n=1}^\infty$
in $\pf(S)$ and any $A\subseteq S$, $\displaystyle d(A)=\lim_{n\to\infty}\sup_{m\geq n}\max_{s\in S}
\frac{|A\cap F_ms|}{|F_m|}$.

\begin{definition}
    Let $(S,\cdot)$ be a semigroup, $H \in \pf(S)$, and $\epsilon > 0$.  
A set $F \in \pf(S)$ is {\it $(H,\epsilon)$-invariant} if for all $h \in H$, $|F \setminus hF| < \epsilon \cdot|F|$.
\end{definition}

\begin{lemma}
\label{lemma_invariant_dt_equality}
Let $(S,\cdot)$ be a semigroup, and let $A \subseteq S$.  For every $\epsilon > 0$, 
there exists $H \in \pf(S)$ such that for every $(H,\epsilon/2)$-invariant set $F \in \pf(S)$,
$$d_t(A) \leq \max_{s \in S} \frac{|As^{-1} \cap F|}{|F|} < d_t(A) + \epsilon.$$
\end{lemma}

\begin{proof}
Let $\epsilon > 0$.  By the definition of $d_t(A)$, there exists $H \in \pf(S)$ such that
$$\max_{s \in S} \frac{|As^{-1} \cap H|}{|H|} < d_t(A) + \frac{\epsilon}{2}\,.$$
Let $F \in \pf(S)$ be $(H,\epsilon/2)$-invariant, and define
$$v = \max_{s \in S} \frac{|As^{-1} \cap F|}{|F|}.$$
Let $s_0 \in S$ achieve this maximum so that $|A s_0^{-1} \cap F| = v\cdot |F|$.

By the definition of $d_t(A)$, we see that $v \geq d_t(A)$. Therefore, to conclude the
proof of the lemma, we need only show that $v < d_t(A) + \epsilon$.

We claim that for all $h \in H$, $|h^{-1}As_0^{-1} \cap F| > (v - \epsilon/2)\cdot |F|$.
Indeed, let $h \in H$.  Since
$$\big( As_0^{-1} \cap hF \big) \cup \big( F \setminus h F \big) \supseteq As_0^{-1} \cap F\,,$$
we have that $|As_0^{-1} \cap hF| \geq |As_0^{-1} \cap F| - |F \setminus h F|$. 
Since $\lambda_h[h^{-1} As_0^{-1} \cap F] \supseteq As_0^{-1} \cap h F$, 
$$|h^{-1} As_0^{-1} \cap F| \geq |As_0^{-1} \cap h F| 
\geq |As_0^{-1} \cap F| - |F \setminus hF| > \left(v - \frac{\epsilon}{2} \right)\cdot |F|.$$

Now we see that
$$ \begin{array}{rl}
v - \displaystyle \frac{\epsilon}{2} &<  \displaystyle \frac{1}{|H|} \sum_{h \in H} \frac{|h^{-1}As_0^{-1} \cap F|}{|F|}\\
&=  \displaystyle \frac{1}{|H|}\frac{1}{|F|} \sum_{h \in H}  \sum_{f \in F} \cchi_{h^{-1}As_0^{-1}}(f)\\
&= \displaystyle \frac{1}{|F|} \frac{1}{|H|}\sum_{f \in F}  \sum_{h \in H} \cchi_{A(fs_0)^{-1}}(h)\\
&= \displaystyle \frac{1}{|F|} \sum_{f \in F} \frac{|A(fs_0)^{-1} \cap H|}{|H|}.
\end{array}$$
It follows that there exists $f \in F$ such that
$$v- \frac{\epsilon}{2} < \frac{|A(fs_0)^{-1} \cap H|}{|H|} \leq \max_{s \in S} \frac{|As^{-1} \cap H|}{|H|} < d_t(A) + \frac{\epsilon}{2}.$$
This implies that $v < d_t(A) + \epsilon$, as was to be shown.
\end{proof}

\begin{theorem}
\label{theorem_density_by_arbitrary_folner_net}
    Let $(S,\cdot)$ be a semigroup satisfying SFC, and 
let $\langle F_\alpha \rangle_{\alpha \in D}$ be a F\o{}lner net in $\pf(S)$.  For all $A \subseteq S$,
    $$d(A) = d^*(A) = d_t(A) = \lim_{\alpha \in D} \max_{s \in S} \frac{|As^{-1} \cap F_\alpha|}{|F_\alpha|}\,.$$
\end{theorem}

\begin{proof}
    The first two equalities follow from Theorem \ref{d=dstar}.  It follows from the definition of a F\o lner net
that for any $H\in\pf(S)$ and any $\epsilon>0$ there exists $\alpha\in D$ such that for all $\sigma\geq\alpha$,
$F_\sigma$ is $(H,\epsilon)$-invariant, so the last equality follows from 
Lemma \ref{lemma_invariant_dt_equality}.
\end{proof}

We extend the result of \cite[Corollary 3.6]{BG} to uncountable semigroups. 

\begin{corollary}\label{corrtcanc} Let $(S,\cdot)$ be a right cancellative semigroup satisfying SFC, and 
let $\langle F_\alpha \rangle_{\alpha \in D}$ be a F\o{}lner net in $\pf(S)$.  For all $A \subseteq S$,
    $$d(A) = d^*(A) = d_t(A) = \lim_{\alpha \in D} \max_{s \in S} \frac{|A \cap F_\alpha s|}{|F_\alpha|}\,.$$
\end{corollary}

\begin{proof} Since $S$ is right cancellative, for any $A\subseteq S$, any $F\in\pf(S)$, and any $s\in S$,
$|A\cap Fs|=|As^{-1}\cap F|$. \end{proof}

In the absence of right cancellation the conclusion of Corollary \ref{corrtcanc} can fail badly.
For example, let $(S,\cdot)$ be an infinite right zero semigroup, and let 
$\emp\neq A\subseteq S$. By Theorem \ref{thmdirect}, $d(A)=1$.   
Any net in $\pf(S)$ is a F\o lner net. If $\langle F_\alpha\rangle_{\alpha\in D}$
is a net in $\pf(S)$ with each $|F_\alpha|=2$, then 
$\displaystyle \lim_{\alpha\in D}\max_{s\in S}
\frac{|A\cap F_\alpha s|}{|F_\alpha|}=\frac{1}{2}$. 
 On the other hand,
$\displaystyle\lim_{\alpha\in D}\max_{s\in S}
\frac{|A\cap F_\alpha s|}{|F_\alpha s|}=1$. So one can
ask what conditions short of right cancellation guarantee that
$d(A)=\displaystyle\lim_{\alpha\in D}\max_{s\in S}
\frac{|A\cap F_\alpha s|}{|F_\alpha s|}$. The only positive answer that
we have is if $d(A)=1$.

\begin{theorem}\label{thmone} Let $(S,\cdot)$ be a semigroup satisfying SFC, and let
$A\subseteq S$. If $d(A)=1$, then for every F\o lner net $\langle F_\alpha\rangle_{\alpha\in D}$
in $\pf(S)$, $$d(A)=\lim_{\alpha\in D}\max_{s \in S}\frac{|A\cap F_\alpha s|}{|F_\alpha s|}\,.$$
\end{theorem}

\begin{proof} Assume that $d(A)=1$, and let $\langle F_\alpha\rangle_{\alpha\in D}$ be a F\o lner net in $\pf(S)$. 
Then by Theorem \ref{thick}, $A$ is thick.
Given $\alpha\in D$, pick $s\in S$ such that $F_\alpha s\subseteq A$. 
Then $\displaystyle\frac{|A\cap F_\alpha s|}{|F_\alpha s|}=1$.
\end{proof}

We see now that we cannot add the case that $d(A)=0$ to the 
statement of Theorem \ref{thmone}.
In this theorem, we deal with the semigroup $(\pf(\ben),\cup)$. In a 
semigroup $(S,\cdot)$, if $A\subseteq S$ and $x\in S$, we write
$A\cdot x$ for $\{y\cdot x:y\in A\}$. If ${\mathcal A}\subseteq\pf(\ben)$
and $M\in\pf(\ben)$, then ${\mathcal A}\cup M$ already means something,
so we write out what we intend, i.e. $\{A\cup M:A\in{\mathcal A}\}$.  Note that
because the semigroup $(\pf(\ben),\cup)$ is commutative, it satisfies SFC.

\begin{theorem}\label{thmNotzero} In the semigroup $(\pf(\ben),\cup)$, 
let ${\mathcal A}=\{X \in \pf(\ben):1 \notin X\}$.
Then $d({\mathcal A})=0$, but there is a F\o lner sequence 
$\langle{\mathcal F}_n\rangle_{n\in\ben}$ in $\pf\big(\pf(\ben)\big)$ such
that $$\displaystyle\lim_{n\in\ben}\max_{T \in \pf(\ben)}
\frac{|{\mathcal A}\cap\{Z \cup T:Z \in {\mathcal F}_n\}|}{|\{Z \cup T:Z \in {\mathcal F}_n\}|}=\frac{1}{2}\,.$$
\end{theorem}

\begin{proof} By Theorem \ref{impliesSFC}, the minimal left ideals of $\beta\pf(S)$ are singletons.
Define $\prec$ on $\pf(S)$ by $A\prec B$ if and only if $A\subseteq B$. Since ${\mathcal A}$ is 
not cofinal in $\pf(S)$, by Theorem \ref{thmdirect}, $d({\mathcal A})=0$.

For $n\in\ben$, let ${\mathcal F}_n=\big\{\{2,3,\ldots,2n\}\big\}\cup\big\{\nhat{k}:n<k\leq 2n\big\}$.
Note that $|{\mathcal F}_n|=n+1$. We claim that 
$\langle{\mathcal F}_n\rangle_{n=1}^\infty$ is a F\o lner sequence. 

Let $X\in\pf(\ben)$, and let $\epsilon>0$ be given.  Pick $n\in\ben$ such that $n>\max X$
and $\frac{1}{n+1}<\epsilon$. Let $m\geq n$. 
Then ${\mathcal F}_m\setminus\{X\cup Z:Z\in{\mathcal F}_m\}\subseteq \big\{\{2,3,\ldots 2m\}\big\}$
so $|{\mathcal F}_m\setminus\{X\cup Z:Z\in{\mathcal F}_m\}|\leq 1 <\epsilon\cdot(n+1)\leq \epsilon\cdot(m+1)
=\epsilon\cdot|{\mathcal F}_m|$.

Now let $n\in\ben$. We shall show that 
$\displaystyle\max_{T \in \pf(\ben)}
\frac{|{\mathcal A}\cap\{Z \cup T:Z \in {\mathcal F}_n\}|}{|\{Z \cup T:Z \in {\mathcal F}_n\}|}=\frac{1}{2}\,.$
So let $T\in\pf(\ben)$. If $1\in T$, then
${\mathcal A}\cap\{Z \cup T:Z \in {\mathcal F}_n\}=\emp$, so assume that
$1\notin T$. Then $\{2,3,\ldots,2n\}\cup T\in{\mathcal A}$ and for $n<k\leq 2n$, 
$\{1,2,\ldots,2n\}\cup T\notin {\mathcal A}$ so $|{\mathcal A}\cap\{Z \cup T:Z \in {\mathcal F}_n\}|=1$.
Also, $\{2,3,\ldots,2n\}\cup T\neq \{1,2,\ldots,2n\}\cup T$ so 
$\displaystyle\frac{|{\mathcal A}\cap\{Z \cup T:Z \in {\mathcal F}_n\}|}{|\{Z \cup T:Z \in {\mathcal F}_n\}|}\leq \frac{1}{2}$.
And if $T=\{2,3,\ldots,2n\}$, then 
$\displaystyle\frac{|{\mathcal A}\cap\{Z \cup T:Z \in {\mathcal F}_n\}|}{|\{Z \cup T:Z \in 
{\mathcal F}_n\}|}=\frac{1}{2}$.
\end{proof}

\section{Density in Product Spaces}

It has been known for some time that the product of two left amenable semigroups is left amenable. This follows from a more powerful theorem, due to Maria Klawe, 
about the semidirect product of two left amenable semigroups \cite[Proposition 3.4]{K}. We prove this directly in Theorem \ref{prodleftamen}, then establish in Theorem \ref{dstarAtimesB} a product property that generalizes \cite[Theorem 3.4]{HSa}.

\begin{theorem}\label{prodleftamen} Let $(S,\cdot)$ and $(T,\cdot)$ be left amenable semigroups, and let $\mu\in LIM(S)$ and $\nu\in LIM(T)$.
Then there exists $\rho\in LIM(S\times T)$ with the property that, for every $A\subseteq S$ and every $B\subseteq T$,
$\rho(A\times B)=\mu(A)\nu(B)$. \end{theorem}

\begin{proof} Let $f\in l_{\infty}(S\times T)$. For each $s\in S$, 
we define $f_s\in l_{\infty}(T)$ by $f_s(t)=f(s,t)$.
Then $\langle\nu(f_s)\rangle_{s\in S}\in l_{\infty}(S)$.
Put $\tau(f)=\mu(\langle\nu(f_s)\rangle_{s\in S})$. For $A\subseteq S$ and $B\subseteq T$,
we claim that $\tau(A\times B)=\mu(A)\nu(B)$. That is,
$\tau(\cchi_{A\times B})=\mu(\cchi_A)\nu(\cchi_B)$. To see this
let $f=\cchi_{A\times B}$ and let $g=\langle\nu(f_s)\rangle_{s\in S}$.
For $s\in S$ and $t\in T$, $f_s(t)=\cchi_B(t)$ if $s\in A$ and 
$f_s(t)=0$ if $s\notin A$. So for $s\in S$, $g(s)=\nu(f_s)=\nu(\cchi_B)\cchi_A(s)$
so $g=\nu(\cchi_B)\cchi_A$. Thus $\tau(\cchi_{A\times B})=\mu(g)=\nu(\cchi_B)\mu(\cchi_A)$.

We claim that $\tau$ is a left invariant on $l_{\infty}(S \times T)$. 
It is clear that $\tau$ is a positive linear functional on $l_{\infty}(S\times T)$.
Since $\tau(\cchi_{S\times T})=1$, $\tau$ is a mean by Lemma \ref{inMNS}.
To see that $\tau$ is left invariant, let $f\in l_{\infty}(S\times T)$, and let $(a,b) \in S\times T$.
In the notation from the previous paragraph, $(f \circ \lambda_{(a,b)})_s = f_{as}\circ \lambda_b$.
Therefore, by the left invariance of $\mu$ and $\nu$,
$$ \begin{array}{rl} \tau(f \circ \lambda_{(a,b)}) &= \mu(\langle\nu((f \circ \lambda_{(a,b)})_s)\rangle_{s\in S}) \\
&= \mu(\langle\nu(f_{as} \circ \lambda_b)\rangle_{s\in S}) \\
&= \mu(\langle\nu(f_{as})\rangle_{s\in S})\\
&= \mu(\langle\nu(f_{s})\rangle_{s\in S}\circ\lambda_a)\\
&= \mu(\langle\nu(f_{s})\rangle_{s\in S})\\
&= \tau(f),\\
\end{array} $$
demonstrating the left invariance of $\tau$.
\end{proof}

\begin{theorem}\label{dstarAtimesB} Let $(S,\cdot)$ and $(T,\cdot)$ be 
left amenable semigroups, let $A\subseteq S$, and let $B\subseteq T$. 
Then $d^*(A\times B)=d^*(A)d^*(B)$.\end{theorem}

\begin{proof} It is a consequence of Theorem \ref{prodleftamen}
that $d^*(A\times B)\geq d^*(A)d^*(B)$, so suppose that 
$d^*(A)d^*(B)<d^*(A\times B)$ and pick $\eta$ such that
$d^*(A)d^*(B)<\eta <d^*(A\times B)$. Pick $\rho\in LIM(S\times T)$
such that $\rho(A\times B)>\eta$. Note in particular that 
$\rho(A\times T)\geq \rho(A\times B)>0$.

We define $\mu$ and $\nu$ mapping ${\mathcal P}(S)$ and ${\mathcal P}(T)$, 
respectively, to $\ber$ by first putting $\mu(X)=\rho(X\times T)$ for every $X\subseteq S$, 
and $\nu(Y)=\displaystyle\frac{\rho(A\times Y)}{\rho(A\times T)}$ for every
$Y\subseteq T$. These functions  are finitely additive on 
${\mathcal P}(S)$ and ${\mathcal P}(T)$ respectively, and 
$\mu(S)=\nu(T)=1$. By Lemma \ref{content}, they extend uniquely to means on $S$ and $T$, respectively.

We claim that these means are left invariant. 
To see this, observe that, for every $s\in S$, $t\in T$, $X\subseteq S$, and $Y\subseteq T$,
$$\mu(s^{-1}X)=\rho(s^{-1}X\times T)=\rho((s,t)^{-1}(X\times T))=\rho(X\times T)=
\mu(X)$$ and
$$\rho(A\cap t^{-1}Y)=\rho\big((s,t)^{-1}(A\cap t^{-1}Y)\big)
= \rho\big((s,t^2)^{-1}(A\cap Y)\big)=\rho(A\cap Y),$$
whereby $\displaystyle\nu(t^{-1}Y)=\frac{\rho(A\times t^{-1}Y)}{\rho(A\times T)}= 
\frac{\rho(A\times Y)}{\rho(A\times T)}=\nu(Y)$.
So, by Lemma \ref{s^{-1}A}, $\mu$ and $\nu$ are left invariant means.

Then $d^*(A)d^*(B)\geq \mu(A)\nu(B)=\rho(A\times B)>\eta$, a contradiction.
\end{proof}

\bibliographystyle{plain}

\end{document}